\documentclass[12pt,reqno]{amsart}

\usepackage[utf8]{inputenc}
\usepackage[T1]{fontenc}
\usepackage{amsfonts,amssymb}
\usepackage{microtype}
\usepackage[colorlinks=true, allcolors=blue]{hyperref}
\usepackage[margin=1in]{geometry}
\usepackage{mathrsfs}
\usepackage{graphicx}

\theoremstyle{plain}
\newtheorem{lemma}{Lemma}
\newtheorem{theorem}{Theorem}

\newtheorem{corollary}{Corollary}

\theoremstyle{definition}
\newtheorem{remark}{Remark}

\newcommand{\T}{\mathbb T}
\newcommand{\ZM}{\mathbb Z_M}
\newcommand{\norm}[1]{\left\lVert#1\right\rVert}
\newcommand{\rank}{\operatorname{rank}}

\begin{document}

\title[Nonrigidity of the trigonometric system]{Nonrigidity of the trigonometric system in $L_p$\\ for $1\le p<2$}
\author{Arthur Sinai}
\date{}

\sloppy

\begin{abstract}
We prove that the trigonometric system $\mathcal{E}_N = \{e_k\}_{|k| \le N}$ is not rigid in $L_p(\mathbb{T})$ for $1 \le p < 2$. That is, as $N\to\infty$, all its elements are approximated with error $o(1)$ by a subspace of dimension $o(N)$.
\end{abstract}

\maketitle

\section{Introduction and Preliminaries}

The Kolmogorov width of order $n$ of a bounded set $W$ in a normed space $X$ is the quantity
\[
d_n(W,X)=\inf_{L\subset X}\ \sup_{x\in W}\ \inf_{y\in L}\norm{x-y}_X,
\]
where the infimum is taken over the linear subspaces $L$ of $X$ of dimension at most $n$.

A system of vectors of a normed space is said to be rigid if all its elements can be approximated with small error only by subspaces whose dimension is of the order of the number of vectors (see~\cite{1}--\cite{3}). Let $\T=\mathbb R/\mathbb Z$ be equipped with Lebesgue measure normalized by $\int_{\T}dt=1$; let $L_p(\T)$, $1\le p<\infty$, be the space of measurable complex-valued functions with finite norm $\norm{f}_{L_p(\T)}=\bigl(\int_{\T}|f|^p\bigr)^{1/p}$; and let $e(t)=e^{2\pi it}$, $e_k(t)=e(kt)$ and $\mathcal E_N=\{e_k\mid|k|\le N\}\subset L_p(\T)$.

In~\cite[\S~6]{1} the question is raised whether the trigonometric system can be approximated in $L_1(\T)$ with error o(1) by subspaces of dimension o(N):
\[
d_n\bigl(\mathcal E_N,L_1(\T)\bigr)=o(1)\quad\text{for a suitable }n=o(N),\ N\to\infty\,?
\]

It is well known that in $L_2$ any orthonormal system of $N$ functions is rigid, namely, $d_n(\{\varphi_1,\ldots,\varphi_N\},L_2)=\sqrt{1-n/N}$ (see, for example, the survey~\cite{4}). In weaker metrics orthogonality alone is no longer enough, but rigidity can be provided by probabilistic properties of the system. The independence of the functions implies rigidity in $L_1$ and even in $L_0$~\cite[Theorems~1.1 and~1.2]{1}, and the same is true for unconditional sets and vectors~\cite{3}; the classical example of an independent system is the Rademacher system~\cite{5}. Lacunary systems behave like independent ones~\cite{6} and are therefore rigid as well~\cite[Statement~3.3]{1}, while for $1<p<2$ the weaker condition $S_{p'}$ suffices~\cite[Theorem~1.3]{1}. The matrix version of the same problem comes from complexity theory. The rigidity of a matrix, that is, the distance in the Hamming metric to the matrices of small rank, was introduced by Valiant in 1977 in connection with lower bounds for the complexity of linear circuits (see the survey~\cite{7}). The unexpected result of Alman and Williams on the nonrigidity of the Walsh--Hadamard matrices~\cite{8} made it possible to approximate the Walsh system in $L_p$ for $p<2$ with a polynomial saving in dimension~\cite[Theorem~1]{2}, while Dvir and Liu proved that the discrete Fourier transform matrix is not rigid either~\cite{9}.

Throughout what follows $p\in[1,2)$ and $\alpha=\frac1p-\frac12\in\bigl(0,\frac12\bigr]$, where the dependence on $p$ occurs only through $\alpha$; we put $\theta=\frac1{18}\sqrt{\alpha\log N}$. We prove (Theorem~\ref{thm:main}) that for large $N$ there is a dimension $n\le2Ne^{-\theta}$ for which $d_n(\mathcal E_N,L_p(\T))\le\frac8\alpha e^{-\theta}$; since $\theta\to\infty$, while the width is nonincreasing in $n$, this gives a positive answer to the above question at once for the whole range $1\le p<2$. The known lower bound rules out good approximation only for $n\le c\log N$~\cite[\S~6]{1}, so the question of the behaviour in the range between $\log N$ and $Ne^{-\theta}$ remains open.

The approximating subspace is provided by the following arithmetic construction. A square-free modulus $M$ with a large number of sufficiently large prime divisors is taken, the Fourier matrix $F_M$ is decomposed, by the Chinese remainder theorem, into a tensor product of matrices $F_q$ with prime $q$, and a suitable $\lambda_qI$ is subtracted from each factor $F_q$. The rank of a factor thereby drops by at least a quarter, whereas the error of the approximation is only $q^{-\alpha}$ and is small precisely when $p<2$; hence the rank decreases exponentially in the number of factors while the total error stays small.

The decomposition of the Fourier matrix into a tensor product of factors of pairwise coprime orders, with a reduction of the rank of each factor, goes back to~\cite[Lemmas~4.8 and~4.9]{9}, where the nonrigidity of the discrete Fourier transform matrix was proved in the sense of Valiant, that is, in the Hamming metric. Here, however, the rank of a factor is reduced in a different way, adapted to the $\ell_p$-norm, which makes it possible to work with arbitrary sufficiently large prime divisors instead of the special choice made in~\cite[\S~4]{9}. For the Walsh system, as noted above, a stronger, polynomial estimate is known, and the same distinction between these two systems is visible in the problem of matrix rigidity.

\section{Main Results}

We carry out the reduction to a discrete problem by means of step functions. For a positive integer $M$ consider the space $\mathbb C^{\ZM}$, $\ZM=\mathbb Z/M\mathbb Z$, with the averaged norm 
\[
\norm{u}_{\ell_p(\ZM)}=\Bigl(\frac1M\sum_{b\in\ZM}|u_b|^p\Bigr)^{1/p}
\]
and the system of characters $\mathcal E^M=\{\varepsilon_a\mid a\in\ZM\}$, where $\varepsilon_a=\bigl(e(ab/M)\bigr)_{b\in\ZM}$. The superscript distinguishes the discrete system from the continuous system $\mathcal E_N$. The normalizations in $L_p(\T)$ and in $\ell_p(\ZM)$ are consistent with each other, namely, the function identically equal to one and the vector all of whose coordinates are equal to one have norm $1$.

\begin{lemma}\label{lem:step}
For all integers $M\ge1$, $N\ge0$ and $n\ge0$
\[
d_n\bigl(\mathcal E_N,L_p(\T)\bigr)
\le d_n\bigl(\mathcal E^M,\ell_p(\ZM)\bigr)+\frac{2\pi N}M.
\]
\end{lemma}

\begin{proof}
We partition $\T$ into $M$ cells $I_b=[b/M,(b+1)/M)$, $b=0,\ldots,M-1$, and define a linear operator $E$ from $\ell_p(\ZM)$ to $L_p(\T)$ by setting $(Eu)(t)=u_b$ for $t\in I_b$. Since $|I_b|=1/M$, we have $\norm{Eu}_{L_p(\T)}^p=\frac1M\sum_b|u_b|^p$, that is, $E$ is an isometry and, in particular, an injection. For $t\in I_b$ and $|k|\le N$
\[
|e_k(t)-e(kb/M)|\le2\pi|k|\,|t-b/M|\le\frac{2\pi N}M,
\]
while the values $e(kb/M)$, $b\in\ZM$, form precisely the set of coordinates of the vector $\varepsilon_{k\bmod M}$; consequently,
\begin{equation}\label{eq:step-err}
\norm{e_k-E\varepsilon_{k\bmod M}}_{L_p(\T)}\le\frac{2\pi N}M.
\end{equation}
Take an arbitrary $\delta>d_n(\mathcal E^M,\ell_p(\ZM))$ and a subspace $\mathcal L\subset\ell_p(\ZM)$ of dimension $\le n$ with $\sup_a\inf_{y\in\mathcal L}\norm{\varepsilon_a-y}_{\ell_p(\ZM)}\le\delta$. Put $L=E(\mathcal L)$; then $\dim L\le n$, and, since $E$ is an isometry, we obtain from~\eqref{eq:step-err}, for every $|k|\le N$,
\[
\inf_{g\in L}\norm{e_k-g}_{L_p(\T)}
\le\frac{2\pi N}M+\inf_{y\in\mathcal L}\norm{E\varepsilon_{k\bmod M}-Ey}_{L_p(\T)}
\le\frac{2\pi N}M+\delta.
\]
Hence $d_n(\mathcal E_N,L_p(\T))\le2\pi N/M+\delta$, and it remains to let $\delta$ tend to $d_n(\mathcal E^M,\ell_p(\ZM))$.
\end{proof}

From now on we work with the discrete problem. The rows of the Fourier matrix $F_q=\bigl(e(ab/q)\bigr)_{a,b\in\mathbb Z_q}$, taken without normalization, are precisely the characters $\varepsilon_a$, and all of them have norm $1$ in $\ell_p(\mathbb Z_q)$, since all the entries have modulus $1$. We denote the identity matrix by $I$.

Everywhere the norm of a matrix is understood row-wise: for 
$A=(A_{ab})_{a,b\in\mathbb Z_q}$ we set
\begin{equation}\label{eq:matrix-norm}
\norm{A}_{\ell_p(\mathbb Z_q)}
=\max_{a\in\mathbb Z_q}\norm{(A_{ab})_{b\in\mathbb Z_q}}_{\ell_p(\mathbb Z_q)}.
\end{equation}
For a matrix consisting of a single row the right-hand side is exactly the norm of this row regarded as a vector, so the notation agrees with the previous one and we use one and the same symbol throughout. This is indeed a norm, being the maximum of the norms of the rows. In particular, the triangle inequality holds for it. It does not change under permutations of rows and columns, and what was said above means that $\norm{F_q}_{\ell_p(\mathbb Z_q)}=1$. We first reduce the rank of the matrix $F_q$, retaining a small error in this norm. The primality of the number $q$ is not used in the next lemma.

\begin{lemma}\label{lem:block}
For every integer $q\ge2$ there exists a matrix $B_q$ of order $q$ such that
\[
\rank B_q\le\frac{3q}4,
\qquad
\norm{F_q-B_q}_{\ell_p(\mathbb Z_q)}=q^{-\alpha},
\qquad
\norm{B_q}_{\ell_p(\mathbb Z_q)}\le1+q^{-\alpha}.
\]
\end{lemma}

\begin{proof}
It is verified directly that $F_q^2=qR$, where $R$ is the matrix of the permutation $a\mapsto-a$ of $\mathbb Z_q$; hence $F_q^4-q^2I=0$. The polynomial $x^4-q^2$ has no multiple roots, so $F_q$ is diagonalizable, and its eigenvalues lie in the set $\{\pm\sqrt q,\pm i\sqrt q\}$. The space $\mathbb C^{\mathbb Z_q}$ is the direct sum of the corresponding four eigenspaces, so by the pigeonhole principle, for one of these four numbers, which we denote by $\lambda_q$, the dimension of the corresponding eigenspace is at least $q/4$. Put $B_q=F_q-\lambda_qI$. Then $\dim\ker B_q\ge q/4$, that is, $\rank B_q\le3q/4$. Each row of the difference $F_q-B_q=\lambda_qI$ contains exactly one nonzero entry, and its modulus is $\sqrt q$, so all the rows of this difference have one and the same norm and
\[
\norm{F_q-B_q}_{\ell_p(\mathbb Z_q)}=\norm{\lambda_qI}_{\ell_p(\mathbb Z_q)}
=\Bigl(\frac{q^{p/2}}q\Bigr)^{1/p}=q^{\frac12-\frac1p}=q^{-\alpha}.
\]
The third estimate follows from the triangle inequality and the equality $\norm{F_q}_{\ell_p(\mathbb Z_q)}=1$.
\end{proof}

The matrices $F_q$ are glued together into a single modulus by the Chinese remainder theorem. If $M_0=M_1M_2$ and $(M_1,M_2)=1$, then, taking integers $c_1,c_2$ with $c_1M_2+c_2M_1=1$ and setting $a_i=a\bmod M_i$, $b_i=c_ib\bmod M_i$, we obtain from $1/M_0=c_1/M_1+c_2/M_2$ that 
\[
e(ab/M_0)=e(a_1b_1/M_1)\,e(a_2b_2/M_2).
\]
The substitution $a\mapsto(a_1,a_2)$ is bijective by the Chinese remainder theorem, and the substitution $b\mapsto(b_1,b_2)$ is bijective because it follows from $c_1M_2\equiv1\pmod{M_1}$ and $c_2M_1\equiv1\pmod{M_2}$ that the residues $c_1$ and $c_2$ are invertible modulo $M_1$ and $M_2$, respectively. Hence $F_{M_0}$ is obtained from $F_{M_1}\otimes F_{M_2}$ by permutations of rows and columns, and the same is true, by induction, for any number of pairwise coprime factors. We shall call the new indices tensor coordinates. The averaged norm is multiplicative with respect to the tensor product: for vectors, 
\[
\norm{u\otimes v}_{\ell_p(\mathbb Z_{M_1M_2})}
=\norm{u}_{\ell_p(\mathbb Z_{M_1})}\norm{v}_{\ell_p(\mathbb Z_{M_2})},
\]
as is seen from $\frac1{M_1M_2}\sum_{a,b}|u_av_b|^p =\bigl(\frac1{M_1}\sum_a|u_a|^p\bigr)\bigl(\frac1{M_2}\sum_b|v_b|^p\bigr)$. The rows of the matrix $A_1\otimes A_2$ are precisely the tensor products of the rows of the matrix $A_1$ with the rows of the matrix $A_2$, and every such pair occurs, so taking the maximum over the rows gives the same for the matrix norm~\eqref{eq:matrix-norm}:
\begin{equation}\label{eq:mult}
\norm{A_1\otimes A_2}_{\ell_p(\mathbb Z_{M_1M_2})}
=\norm{A_1}_{\ell_p(\mathbb Z_{M_1})}\norm{A_2}_{\ell_p(\mathbb Z_{M_2})}.
\end{equation}

\begin{lemma}\label{lem:tensor}
Let the integers $q_1,\ldots,q_m\ge2$ be pairwise coprime and let $M=q_1\cdots q_m$. Then there exists a matrix $B$ of order $M$ such that $\rank B\le(3/4)^mM$ and
\begin{equation}\label{eq:tensor-error}
\norm{F_M-B}_{\ell_p(\ZM)}
\le\prod_{j=1}^m\bigl(1+q_j^{-\alpha}\bigr)-1.
\end{equation}
\end{lemma}

\begin{proof}
In tensor coordinates put $B'=\bigotimes_{j=1}^mB_{q_j}$ with the matrices of Lemma~\ref{lem:block}, and bring $B'$ back to the usual indexing by the same permutations of rows and columns; this yields $B$. Permutations change neither the rank nor the norm. The rank of a tensor product is the product of the ranks, whence the first estimate. Further, the telescoping expansion gives
\[
\bigotimes_{j}F_{q_j}-\bigotimes_{j}B_{q_j}
=\sum_{j=1}^m\Bigl(\bigotimes_{i<j}B_{q_i}\Bigr)\otimes(F_{q_j}-B_{q_j})
\otimes\Bigl(\bigotimes_{i>j}F_{q_i}\Bigr).
\]
By Lemma~\ref{lem:block} and~\eqref{eq:mult} the norm of the $j$th summand does not exceed $\bigl(\prod_{i<j}(1+q_i^{-\alpha})\bigr)q_j^{-\alpha}$, since $\norm{F_{q_i}}_{\ell_p(\mathbb Z_{q_i})}=1$. Summing by the triangle inequality and using the identity
\[
\sum_{j=1}^m\Bigl(\prod_{i<j}(1+x_i)\Bigr)x_j=\prod_{j=1}^m(1+x_j)-1,
\]
we obtain~\eqref{eq:tensor-error}.
\end{proof}

This immediately implies an estimate for the width of the discrete system.

\begin{corollary}\label{cor:discrete}
Let $M$ be the product of pairwise coprime numbers $q_1,\ldots,q_m\ge2$. Then
\begin{equation}\label{eq:discrete-width}
d_{\lfloor(3/4)^mM\rfloor}\bigl(\mathcal E^M,\ell_p(\ZM)\bigr)
\le\prod_{j=1}^m\bigl(1+q_j^{-\alpha}\bigr)-1
\le\exp\Bigl(\sum_{j=1}^mq_j^{-\alpha}\Bigr)-1.
\end{equation}
\end{corollary}

\begin{proof}
Denote by $B$ the matrix of Lemma~\ref{lem:tensor} and by $\mathcal L$ the linear span of its rows, so that $\dim\mathcal L=\rank B\le\lfloor(3/4)^mM\rfloor$, since the rank is an integer. Row $a$ of the matrix $F_M$ is the vector $\varepsilon_a$, while row $a$ of the matrix $B$ lies in $\mathcal L$ and differs from $\varepsilon_a$, in the norm of $\ell_p(\ZM)$, by no more than the left-hand side of~\eqref{eq:tensor-error}, since by definition~\eqref{eq:matrix-norm} the norm of every row of the matrix $F_M-B$ does not exceed the norm of this matrix itself. Therefore the distance from $\varepsilon_a$ to $\mathcal L$ does not exceed this quantity. The last estimate in~\eqref{eq:discrete-width} follows from the inequality $1+x\le e^x$.
\end{proof}

It remains to choose a modulus $M$ of the required size with a large number of sufficiently large prime divisors. Here we use Bertrand's postulate~\cite[\S~22.3]{10}, according to which for any integer $h\ge1$ there exists a prime $q$ with $h<q\le 2h$.

\begin{lemma}\label{lem:modulus}
Let an integer $m\ge2$, an integer $Q\ge3$ and a real number $Z>0$ satisfy the condition
\begin{equation}\label{eq:budget}
2^{m^2}Q^m\le Z^{1/10}.
\end{equation}
Then there exists a square-free $M$ with exactly $m$ prime divisors, all of them
greater than $Q$, such that
\[
Z<M\le2Z,
\qquad
\sum_{q\mid M}q^{-\alpha}\le\frac{Q^{-\alpha}}{1-2^{-\alpha}}\le\frac2\alpha\,Q^{-\alpha}.
\]
\end{lemma}

\begin{proof}
By Bertrand's postulate, for $j=1,\ldots,m-1$ we choose a prime $q_j$ in the interval $(2^{j-1}Q,2^jQ]$, which is admissible since $2^{j-1}Q$ is a positive integer. These intervals are pairwise disjoint, so all the $q_j$ are distinct and greater than $Q$. For $P=\prod_{j<m}q_j$ we have $P\le2^{m(m-1)/2}Q^{m-1}\le2^{m^2}Q^m\le Z^{1/10}$. From~\eqref{eq:budget} and the inequalities $m\ge2$, $Q\ge3$ it follows that $Z^{1/10}\ge2^4\cdot3^2>1$, in particular, $Z>1$ and $P\le Z^{1/10}\le Z$. Hence the integer $h=\lfloor Z/P\rfloor$ is at least $1$, and one more application of Bertrand's postulate yields a prime $q_m$ with $h<q_m\le2h$. Since $q_m$ is an integer, this gives
\[
\frac ZP<h+1\le q_m\le2h\le\frac{2Z}P.
\]
Since $q_m>Z/P\ge Z^{9/10}\ge Z^{1/10}\ge2^{m-1}Q$, the number $q_m$ is distinct from all the $q_j$, $j<m$, and is greater than $Q$. Hence $M=Pq_m$ is square-free and has exactly $m$ prime divisors, and $Z<M\le2Z$. Finally, $q_j>2^{j-1}Q$ for all $j\le m$, whence
\[
\sum_{q\mid M}q^{-\alpha}\le Q^{-\alpha}\sum_{j\ge1}2^{-(j-1)\alpha}
=\frac{Q^{-\alpha}}{1-2^{-\alpha}}.
\]
It remains to note that for $0<\alpha\le\frac12$
\[
1-2^{-\alpha}\ge\alpha\log2-\frac{(\alpha\log2)^2}2
\ge\alpha\log2\Bigl(1-\frac{\log2}4\Bigr)>\frac\alpha2.
\]
\end{proof}

\begin{theorem}\label{thm:main}
Let $1\le p<2$, $\alpha=\frac1p-\frac12$ and
\[
\log N\ \ge\ \frac{324}\alpha\max\Bigl(50,\ \log\frac4\alpha\Bigr)^2.
\]
Put $\theta=\frac1{18}\sqrt{\alpha\log N}$. Then there exists an integer
$n=n(N,p)$ for which
\begin{equation}\label{eq:main}
n\le2N e^{-\theta},
\qquad
d_{n}\bigl(\mathcal E_N,L_p(\T)\bigr)\le\frac8\alpha\,e^{-\theta}.
\end{equation}
In particular, for each fixed $p\in[1,2)$
\[
n(N,p)\le2N\exp\Bigl(-\tfrac1{18}\sqrt{\alpha\log N}\Bigr)=o(N),
\qquad
d_{n(N,p)}\bigl(\mathcal E_N,L_p(\T)\bigr)\to0,
\]
that is, the trigonometric system is not rigid in $L_p(\T)$.
\end{theorem}

\begin{proof}
The condition on $N$ means precisely that $\theta\ge\max(50,\log(4/\alpha))$.
Put
\[
m=\lceil8\theta\rceil,
\qquad
Q=\bigl\lceil e^{\theta/\alpha}\bigr\rceil,
\qquad
Z=Ne^{\theta}.
\]
The numbers $m$ and $Q$ are integers, and $m\ge2$, $Q\ge3$.

We first verify condition~\eqref{eq:budget}. Since $\theta\ge50$, we have
$m\le8\theta+1\le8.02\,\theta$, and since $\theta/\alpha\ge\theta\ge50$ and
$\log(x+1)\le\log x+1/x$, we have $\log Q\le\frac\theta\alpha+1\le1.02\frac\theta\alpha$.
Therefore, using $1/\alpha\ge2$,
\[
m^2\log2+m\log Q
\le44.6\,\theta^2+8.19\frac{\theta^2}\alpha
\le\frac{\theta^2}\alpha\bigl(22.3+8.19\bigr)
\le30.5\,\frac{\theta^2}\alpha.
\]
But $\theta^2=\frac{\alpha\log N}{324}$, so the right-hand side does not exceed
$\frac{30.5}{324}\log N<\frac1{10}\log N\le\frac1{10}\log Z$, which is exactly
condition~\eqref{eq:budget}.

Thus, by Lemma~\ref{lem:modulus} there is a square-free $M=q_1\cdots q_m$ with $Z<M\le2Z$ and all $q_j>Q\ge e^{\theta/\alpha}$, and for the sum
$s=\sum_{j=1}^mq_j^{-\alpha}$ we have
\[
s\le\frac2\alpha Q^{-\alpha}\le\frac2\alpha e^{-\theta}\le\frac12,
\]
since $\theta\ge\log(4/\alpha)$. Put $n=\lfloor(3/4)^mM\rfloor$.

Let us estimate this dimension. The inequality $\log x\ge1-1/x$ with $x=4/3$ implies
$\log\frac43\ge\frac14$, so $(3/4)^m\le e^{-m/4}\le e^{-2\theta}$. Together with $M\le2Z=2Ne^{\theta}$ this gives $n\le2Ne^{-\theta}$.

It remains to estimate the width. By Corollary~\ref{cor:discrete} and the inequality
$e^s-1\le se^{s}\le2s$ (valid for $s\le\frac12$)
\[
d_n\bigl(\mathcal E^M,\ell_p(\ZM)\bigr)\le e^s-1\le2s\le\frac4\alpha e^{-\theta}.
\]
Moreover, $2\pi N/M<2\pi N/Z=2\pi e^{-\theta}\le\frac{3.2}\alpha e^{-\theta}$, again because $1/\alpha\ge2$. Lemma~\ref{lem:step} gives
\[
d_n\bigl(\mathcal E_N,L_p(\T)\bigr)
\le\frac4\alpha e^{-\theta}+\frac{3.2}\alpha e^{-\theta}
\le\frac8\alpha e^{-\theta},
\]
which, together with the estimate for the dimension, yields~\eqref{eq:main}.
\end{proof}

\begin{remark}
All the subspaces above are complex; a passage to the real case is available. Let
$L_p(\T;\mathbb R)$ be the space of real-valued functions in $L_p(\T)$, let $d^{\mathbb R}_n$ be
the width with respect to real subspaces of $L_p(\T;\mathbb R)$ of dimension $\le n$, and let
\[
\mathcal T_N=\{\cos2\pi kt\}_{0\le k\le N}\cup\{\sin2\pi kt\}_{1\le k\le N}\subset L_p(\T;\mathbb R).
\]
If a complex subspace $Y\subset L_p(\T)$ with $\dim_{\mathbb C}Y\le n$ approximates $\mathcal E_N$
with error $\delta$, then $Y_{\mathbb R}=\operatorname{Re}Y$ is a real subspace of
$L_p(\T;\mathbb R)$ with $\dim_{\mathbb R}Y_{\mathbb R}\le\dim_{\mathbb R}Y\le2n$ (the map
$g\mapsto\operatorname{Re}g$ is $\mathbb R$-linear), and it approximates $\mathcal T_N$ with the
same error: indeed
\[
\cos2\pi kt=\operatorname{Re}e_k,
\qquad
\sin2\pi kt=\operatorname{Re}(-ie_k),
\]
while $g$ and $-ig$ run over $Y$ simultaneously, and
$\norm{\operatorname{Re}f-\operatorname{Re}g}_{L_p(\T)}\le\norm{f-g}_{L_p(\T)}$ because
$|\operatorname{Re}z|\le|z|$. Hence
\begin{equation}\label{eq:real}
d^{\mathbb R}_{2n}\bigl(\mathcal T_N,L_p(\T;\mathbb R)\bigr)\le d_n\bigl(\mathcal E_N,L_p(\T)\bigr),
\end{equation}
so Theorem~\ref{thm:main} also holds for $\mathcal T_N$ with $2n$ in place of $n$.
\end{remark}

\textbf{AI disclosure statement.} Kimi K3 was used as a discussion partner. The starting point, namely the decision to work with the discrete Fourier matrix, was the author's and was prompted by the paper of Dvir and Liu. The specific plan of the proof took shape in discussion with the model. All proofs were then carried out and verified by the author, who takes full responsibility for the content of this note.

\end{document}